\documentclass{amsart}
 \newtheorem{theorem}{Theorem}[section]
 \newtheorem{corollary}[theorem]{Corollary}
 
 \newtheorem{proposition}[theorem]{Proposition}
 \theoremstyle{definition}
 \newtheorem{definition}[theorem]{Definition}
 \theoremstyle{remark}
 \newtheorem{rem}[theorem]{Remark}
 
 \numberwithin{equation}{section}

\usepackage{xcolor}
\usepackage[linkbordercolor=cyan, citebordercolor=cyan]{hyperref}

\def\d{{\rm d}}

\def\R{{\mathbb R}}

\DeclareMathOperator\arccot{arccot}

\begin{document}

\title[On the higher derivates of arctan]
 {On the higher derivates of arctan}

\author[Oliver Deiser]{Oliver Deiser}

\address{Zentrum Mathematik and School of Education\\
Technische Universit\"at M\"unchen\\
80290 M\"unchen, Germany}

\email{deiser@tum.de}

\author{Caroline Lasser}
\address{Zentrum Mathematik\\
Technische Universit\"at M\"unchen\\
80290 M\"unchen, Germany}
\email{classer@ma.tum.de}


\date{\today}

\begin{abstract}
We give a rational closed form expression for the higher derivatives of the inverse tangent
function and discuss its relation to Chebyshev polynomials, trigonometric expansions and Appell sequences of polynomials.
\end{abstract}

\maketitle
\section{Introduction}

The higher derivatives of the inverse tangent function are rational functions, 
\[
\arctan^{(n)}(x) = (n-1)! \;\frac{q_{n-1}(x)}{(1+x^2)^n},\qquad n\ge 1,
\]
defining a polynomial sequence $(q_n)_{n\ge0}$ with integer coefficients. These polynomials have scattered appearances in the literature. 
They have been described in terms of Chebyshev polynomials of the first and second kind in the handbook entry 
\cite[Chapter 1.1.7]{Bry} as well as in terms of 
trigonometric functions in \cite[Theorem~1]{AdLa}, see also a sign correction in \cite[Theorem~1]{Lam}. 
A monomial expansion, whose coefficients are specified up to degree five, is given in \cite[Formula~9.3.1]{Kon}.

\smallskip
Our note aims at a unified exposition in establishing the simple closed-form representation
\[
q_n(x) = (-1)^n \sum_{k\,\text{even}, \,0\le k\le n} \binom{n+1}{k+1} (-1)^{k/2} x^{n-k},
\]
relating it to the known expressions in terms of Chebyshev polynomials and trigonometric functions, and proving 
that $(q_n)_{n\ge0}$ forms an Appell sequence if suitably normalized. We 
introduce the arctan polynomials using elementary real analysis based on Pascal's triangle and a three-term 
recurrence relation~(\S\ref{sec:alg} and \S\ref{sec:arctan}). Then, we discuss the polynomials' relation to the Chebyshev 
polynomials (\S\ref{sec:orth}), to Appell sequences (\S\ref{sec:app}) and to trigonometric expansions (\S\ref{sec:trig}). 
In Appendix~\ref{appendix}, we derive them by a partial fraction decomposition.


\section{Algebraic polynomials}\label{sec:alg}
We explicitly calculate the first six derivatives:
\begin{align*}
\arctan'(x) &= \frac{1}{1+x^2},\\*[1ex]
\arctan''(x) &= \frac{-2x}{(1+x^2)^2},\\*[1ex]
\arctan'''(x) &= \frac{2(3x^2-1)}{(1+x^2)^3},\\*[1ex]
\arctan^{(4)}(x) &= \frac{6(-4x^3+4x)}{(1+x^2)^4},\\*[1ex]
\arctan^{(5)}(x) &= \frac{24(5 x^4 - 10x^2 + 1)}{(1+x^2)^5},\\*[1ex]
\arctan^{(6)}(x) &= \frac{120(-6 x^5 + 20 x^3 - 6x)}{(1+x^2)^6}.
\end{align*}
We observe fractions with polynomials in the numerator and powers of the sum $1+x^2$ in the denumerator. Up to a factor $(n-1)!$ 
and alternating signs, the coefficients of these polynomials appear in Pascal's triangle, see Figure~\ref{pascal}. 

\begin{table}[h]
\begin{tabular}{ccccccccccccccccccc}
   & & & & & & & & 1\\
   & & & & & & & 1 & & \textcolor{red}{\textcircled{1}}\\
   & & & & & & 1 & & \textcolor{blue}{\textcircled{2}} & & 1 \\
   & & & & & 1 & & \textcolor{red}{\textcircled{3}} & & 3 & & \textcolor{blue}{\textcircled{1}}\\
   & & & & 1 & & \textcolor{blue}{\textcircled{4}} & & 6 & & \textcolor{red}{\textcircled{4}} & & 1\\
   & & & 1 &  & \textcolor{red}{\textcircled{5}} & & 10 & & \textcolor{blue}{\textcircled{10}} & & 5 & & \textcolor{red}{\textcircled{1}}\\
   & &  1 &  & \textcolor{blue}{\textcircled{6}} & &  15 & & \textcolor{red}{\textcircled{20}} & & 15 & & \textcolor{blue}{\textcircled{6}} & & 1\\*[2ex]
\end{tabular}
\caption{\label{pascal} Pascal's triangle. The encircled coloured  numbers correspond to the coefficients of the six polynomials up to multiplicative factors of the size $(n-1)!$ and alternating signs. \textcolor{red}{Red color} indicates a positive sign, 
\textcolor{blue}{blue color} a negative one.}
\end{table}

This motivates the following definition:

\begin{definition} We define a sequence $(q_n)_{n\ge0}$ of polynomials 
$q_n:\R\to\R$ by
\[
q_n(x) = (-1)^n \sum_{k\,\text{even}, \,0\le k\le n} \binom{n+1}{k+1} (-1)^{k/2} x^{n-k}.
\]
\end{definition}

Before establishing the relation of  these polynomials to the higher derivatives of the inverse tangent function, we collect some properties.

\begin{proposition}[Basic properties]\label{prop} $\, $

\begin{enumerate}
\setlength{\itemsep}{1ex}
\item $q_n$ is a polynomial of degree $n$ for each $n\ge0$.
\item $q_n(-x) = (-1)^n q_n(x)$ for each $n\ge 0$.
\item $q_n(0) = 0$ for $n$ odd and $q_n(0) = (-1)^{n/2}$ for $n$ even. 
\item $q_n' = -(n+1) q_{n-1}$ for each $n\ge 1$.
\item $q_n(x) + 2x q_{n-1}(x) + (1+x^2) q_{n-2}(x) = 0$ for each $n\ge 2$.
\end{enumerate}
\end{proposition}

\begin{proof}
The first four properties are immediate. 
For proving (5), we argue by induction. We check that 
\[
q_2(x) + 2x q_1(x) + (1+x^2)q_0(x)= 0.
\] 
Assuming that the relation holds for $n-1$, we obtain
\begin{align*}
&\frac{\d}{\d x} \left( q_n(x) + 2xq_{n-1}(x) + (1+x^2) q_{n-2}(x) \right)\\
& = q_n'(x) + 2 q_{n-1}(x) + 2xq_{n-1}'(x) + 2x q_{n-2}(x) + (1+x^2)q_{n-2}'(x) \\
&= (1-n) \left( q_{n-1}(x) + 2x q_{n-2}(x) + (1+x^2) q_{n-3}(x)\right)\\
&= 0.
\end{align*}
Hence, the polynomial 
\[
q_n(x) + 2xq_{n-1}(x) + (1+x^2) q_{n-2}(x)
\] 
is constant. By the second property, it evaluates to zero for $x=0$. Therefore it is the zero polynomial.
\end{proof}

\section{Arctan function}\label{sec:arctan}
We are now ready to use the polynomial sequence $(q_n)_{n\ge0}$ for describing the derivatives of the inverse tangent function.

\begin{theorem}[Derivatives of arctan]\label{theo}
\[
\arctan^{(n)}(x) = (n-1)! \;\frac{q_{n-1}(x)}{(1+x^2)^n},\qquad n\ge 1.
\]
\end{theorem}

\begin{proof} We argue by induction on $n$. For $n=1$, the formula reduces to the first derivative of the $\arctan$. 
For the inductive step, we calcuate
\begin{align*}
\arctan^{(n+1)}(x) 
&= (n-1)! \;\frac{q_{n-1}'(x) (1+x^2)^n - 2n \,x\, q_{n-1}(x) (1+x^2)^{n-1} }{(1+x^2)^{2n}} \\
&= (n-1)! \;\frac{-n q_{n-2}(x) (1+x^2) - 2n\, x\, q_{n-1}(x)}{(1+x^2)^{n+1}}\\
&= n! \;\frac{q_n(x)}{(1+x^2)},
\end{align*}
where the last two equations rely on Proposition~\ref{prop}.
\end{proof}

In Appendix~\ref{appendix} we give another proof of Theorem~\ref{theo} using a partial fraction 
decomposition of the rational function $1/(1+x^2)$. 

\begin{rem}[Maclaurin series]
The Maclaurin series of the inverse tangent function,
\[
\arctan(x) =  \sum_{n\ge 0} (-1)^n \frac{x^{2n+1}}{2n + 1},
\]
is usually derived via a geometric series argument. Theorem~\ref{theo} together with the third property of 
Proposition~\ref{prop} provides an alternative derivation. 
\end{rem}

\section{Orthogonal polynomials}\label{sec:orth}
Sequences of real polynomials $(\pi_n)_{ n\ge0}$ that are orthogonal to each other with respect to an inner product 
satisfy a three-term recurrence relation of the form
\[
\pi_n(x) - (\alpha_n x-\beta_n)\pi_{n-1}(x) + \gamma_n \pi_{n-2}(x) = 0,\qquad n\ge2,
\]
where $(\alpha_n)$, $(\beta_n)$ and $(\gamma_n)$ are sequences of real numbers, see \cite[Theorem~3.2.1]{Sze}. Hence, the arctan 
polynomials $(q_n)_{n\ge0}$ are not orthogonal, since for all real numbers $\alpha,\beta,\gamma\in\R$,
\begin{align*}
(\alpha x-\beta)q_3(x) -\gamma q_2(x) &=(\alpha x-\beta)(-4x^3+4x) - \gamma (3x^2 - 1)\\
&\neq 5x^4 - 10x^2 + 1 \\
&=q_4(x).
\end{align*}
However, there is a known relation of inverse tangent derivatives to orthogonal polynomials, namely the Chebyshev polynomials 
of the first and second kind. 

\begin{proposition}[Chebyshev polynomials]\label{cheb}
\begin{align*}
q_{2n}(x) &= (-1)^n (1+x^2)^{n+1/2} \, T_{2n+1}\!\left(\frac{1}{\sqrt{1+x^2}}\right),\qquad n\ge0,\\
q_{2n-1}(x) &= (-1)^n (1+x^2)^{n-1/2} \, x\;  U_{2n-1}\!\left(\frac{1}{\sqrt{1+x^2}}\right),\qquad n\ge 1,
\end{align*}
where $(T_n)_{n\ge0}$ and $(U_n)_{n\ge0}$ denote the $n$th Chebyshev polynomial of the first and second kind, respectively.
\end{proposition}

\begin{proof}
We just verify the formula for the even degree arctan polynomials, since the proof for the odd ones is analogous. 
Using that
\[
T_n(x) = \sum_{k=0}^{\lfloor \frac{n}{2}\rfloor} \binom{n}{2k} (x^2-1)^k x^{n-2k},
\]
we obtain
\[
 T_{2n+1}\!\left(\frac{1}{\sqrt{1+x^2}}\right) 
=  (1+x^2)^{-n-\frac12}  \sum_{k=0}^{n} \binom{2n+1}{2k} (-1)^k x^{2k}.
\]
Therefore,
\begin{align*}
q_{2n}(x) &= \sum_{k\,\text{even}, \,0\le k\le 2n} \binom{2n+1}{k+1} (-1)^{k/2} x^{2n-k}\\
&= (-1)^n \sum_{m=0}^{n} \binom{2n}{2m+1} (-1)^{m} x^{2m}\\
&= (-1)^n (1+x^2)^{n+1/2} \, T_{2n+1}\!\left(\frac{1}{\sqrt{1+x^2}}\right).
\end{align*}
\end{proof}

The two formulas of the following Corollary~\ref{cor} are given in \cite[Chapter 1.1.7]{Bry}. 
They are implied by Theorem~\ref{theo} and Proposition~\ref{cheb}.

\begin{corollary}[Derivatives of arctan]\label{cor}
\begin{align*}
\arctan^{(2n+1)}(x) &= \frac{ (-1)^n (2n)!}{(1+x^2)^{n+1/2}} \,  T_{2n+1}\!\left(\frac{1}{\sqrt{1+x^2}}\right),\qquad n\ge 0,\\*[1ex]
\arctan^{(2n)}(x) &= \frac{ (-1)^n (2n-1)!}{(1+x^2)^{n+1/2}} \, x\;  U_{2n-1}\!\left(\frac{1}{\sqrt{1+x^2}}\right),\qquad n\ge 1.
\end{align*}
\end{corollary}

\section{Appell polynomials}\label{sec:app}

We normalize the polynomial sequence $(q_n)_{n\ge0}$ to obtain monic polynomials
\[
p_n := \frac{(-1)^n}{n+1} q_n,\qquad n\ge 0,
\]
that may also be written as
\[
p_n(x) = \sum_{k\,\text{even},\,0\le k\le n} \binom{n}{k} \frac{(-1)^{k/2}}{k+1} x^{n-k}.
\]
The derivatives of the inverse tangent function given in Theorem~\ref{theo} may be expressed in terms of these polynomials as
\[
\arctan^{(n)}(x) = (-1)^{n-1}\;n!\; \frac{p_{n-1}(x)}{(1+x^2)^n},\qquad n\ge 1.
\]
According to Proposition~2, the derivatives satisfy
\begin{align*}
p_n' &= (-1)^{n-1} q_{n-1} \\
&= n p_{n-1},\qquad\qquad n\ge 1.
\end{align*}
Hence, the monic polynomials $(p_n)_{n\ge0}$ are an example of an {\em Appell sequence}, that due to its elementary connection to the inverse tangent function might deserve addition 
to the compilations of noteworthy Appell polynomials as given in \cite{Car,Sca} or in \cite[Chapter~2.6]{Rom}.

\section{Trigonometric expansions}\label{sec:trig}

The triangular array $(\alpha_{k,n})$ of integers defined by
\[
q_n(x) = (-1)^n \sum_{k=0}^n \alpha_{n,k} \,x^{n-k},\qquad n\ge 0,
\]
also occurs in the context of trigonometric expansions, see \cite[Section~5]{Kim}. 
Examining this connection, we obtain the following trigonometric representation of the arctan polynomials.

\begin{theorem}[Trigonometric representation]\label{trig}
\[
\sin(n\arccot x) = \frac{(-1)^{n-1} \,q_{n-1}(x)}{(1+x^2)^{n/2}},\qquad n\ge1,
\]
where $\arccot(x)= \frac\pi2 - \arctan(x)$ refers to the inverse function of the cotangens restricted to $(0,\pi)$. 
\end{theorem}

\begin{proof} By de Moivre's formula and the binomial theorem, we may write
\[
\cos(nx) + i\sin(nx)  = \sum_{k=0}^n \binom{n}{k} i^k \cos^{n-k}(x) \sin^k(x).
\]
This implies the trigonometric representation
\[
\sin(nx) = \sum_{k\,\text{odd},\,0\le k\le n} \binom{n}{k} (-1)^{(k+1)/2} \cos^{n-k}(x) \sin^k(x).
\]
Since
\[
\sin(\arccot x) = \frac{1}{\sqrt{1+x^2}}\quad\text{and}\quad\cos(\arccot x) = \frac{x}{\sqrt{1+x^2}},
\]
we have
\begin{align*}
\sin(n\arccot x) 
&= \frac{1}{(1+x^2)^{n/2}} \sum_{k\,\text{odd},\,0\le k\le n} \binom{n}{k} (-1)^{(k+1)/2} x^{n-k}\\
&= \frac{(-1)^{n-1} \, q_{n-1}(x)}{(1+x^2)^{n/2}}.
\end{align*}
\end{proof}

Combining Theorem~\ref{theo} and Theorem~\ref{trig}, we obtain a second representation of 
the higher derivatives of the arctan function, that is similar to \cite[Theorem~1]{AdLa} and \cite[Theorem~1]{Lam}.

\begin{corollary}[Derivatives of arctan]
\[
\arctan^{(n)}(x) = (n-1)! \;\frac{(-1)^{n-1} \sin(n\arccot x)}{(1+x^2)^{n/2}},\qquad n\ge 1.
\]
\end{corollary}

\appendix
\section{Partial fraction decomposition}\label{appendix}
The quadratic polynomial $1+x^2$ has the complex roots $\pm i$ and thus can be factorized as $1+x^2 = (x-i)(x+i)$. This provides the partial fraction decomposition
\[
\frac{1}{1+x^2} = \frac{1}{2i} \left( \frac{1}{x-i} - \frac{1}{x+i}\right).
\]
Therefore, for $n\ge 1$, 
\begin{align*}
\arctan^{(n)}(x) & =\frac{\d^{n-1}}{\d x^{n-1}}\frac{1}{1+x^2} \\
&= (n-1)!\; \frac{(-1)^{n-1}}{2i} \left( \frac{1}{(x-i)^n} - \frac{1}{(x+i)^n}\right)\\
&= (n-1)!\;\frac{(-1)^{n-1}}{2i} \; \frac{(x+i)^{n} - (x-i)^{n}}{(1+x^2)^n}
\end{align*}
By the binomial theorem,
\begin{align*}
\frac{(-1)^{n-1}}{2i}& \left( (x+i)^n - (x-i)^n\right) = \sum_{k=0}^n \binom{n}{k} \frac{i^k - (-i)^k}{2i} x^{n-k}\\
 &= (-1)^{n-1}\sum_{k\,\text{odd},\,0\le k\le n} \binom{n}{k} i^{k-1} x^{n-k}\\
 &= (-1)^{n-1}\sum_{\ell\,\text{even},\,0\le \ell\le n-1} \binom{n}{\ell+1} (-1)^{\ell/2} x^{n-1-\ell}\\*[1ex]
 &= q_{n-1}(x).
\end{align*}

\bigskip
\subsection*{Acknowledgments}
We thank C.~Vignat and F.~Bornemann for pointing us to \cite{Bry} and \cite{AdLa,Lam}, respectively.

\end{document}